\documentclass[12pt]{article}
\usepackage{amsmath,amsthm,amssymb}
\usepackage{amssymb,latexsym}

\newtheorem{theorem}{Theorem}

\newtheorem{lemma}{Lemma}

\begin{document}

\baselineskip=17pt

\title{\bf On an equation with prime numbers close to squares}

\author{\bf S. I. Dimitrov}
\date{2019}
\maketitle
\begin{abstract}
Let $[\, \cdot\,]$ be the floor function.
In this paper, we show that when $1<c<37/36$, then
every sufficiently large positive integer $N$ can be represented in the form
\begin{equation*}
N=[p^c_1]+[p^c_2]+[p^c_3]\,,
\end{equation*}
where $p_1,p_2,p_3$ are primes close to squares.
\quad\\
{\bf  2010 Math.\ Subject Classification}:  11P32 $\cdot$ 11P55
\end{abstract}

\section{Introduction and main result}
\indent

In 1995 Laporta and Tolev  \cite{Laporta-Tolev} considered the diophantine equation
\begin{equation*}
[p^c_1]+[p^c_2]+[p^c_3]=N\,,
\end{equation*}
where $p_1,\, p_2,\, p_3$ are primes.
For $1<c <17/16$ they proved that for the sum
\begin{equation*}
R(N)=\sum\limits_{[p^c_1]+[p^c_2]+[p^c_3]=N}\log p_1\log p_2\log p_3
\end{equation*}
the asymptotic formula
\begin{equation}\label{RNasymptoticformula}
R(N)=\frac{\Gamma^3(1 + 1/c)}{\Gamma(3/c)}N^{3/c-1}
+\mathcal{O}\Big(N^{3/c-1}\exp\big(-(\log N)^{1/3-\varepsilon}\big)\Big)
\end{equation}
holds.

Later the result of Laporta and Tolev was improved by Kumchev and Nedeva \cite{Kumchev-Nedeva}
to $1<c <12/11$, by Zhai and Cao  \cite{Zhai-Cao} to $1<c<258/235$ and finally
by Cai \cite{Cai} to $1<c <137/119$ and this is the best result up to now.

In 1997 Kumchev and Tolev \cite{Kumchev-Tolev}  proved
by asymptotic formula  that when $N_1,\ldots,N_n$
are sufficiently large positive integers then the system
\begin{equation}\label{Kumchev-Tolev-system}
\left|\begin{array}{cccc}
p_1+p_2+\cdots+p_k=N_1\\
p^2_1+p^2_2+\cdots+p^2_k=N_2\\
\cdots\cdots\cdots\cdots\cdots\cdots\cdots\\
p^n_1+p^n_2+\cdots+p^n_k=N_n
\end{array}\right.
\end{equation}
is solvable in primes $p_1,\ldots,p_k$ near to squares,
i.e. such that $\sqrt{p_1},\ldots,\sqrt{p_k}$  are close to integers.
In the system  \eqref{Kumchev-Tolev-system}, $n \geq2$ and $k\geq k_0(n)$, where $k_0(n)$ is defined by the table
\begin{center}
  \begin{tabular}{ | l | c | r | l | c | r | l | c | r | l | c | }
    \hline
    n & 2 & 3 &  4 & 5  &  6 &  7 &  8 &  9 & 10  \\ \hline
$k_0(n )$ & 7 & 19 & 49 & 113 & 243 & 413 & 675 & 1083 & 1773  \\ \hline
  \end{tabular}
\end{center}
in the case of $2\leq n \leq10$, and by the formula
\begin{equation*}
k_0(n)=2[n^2(3 \log n + \log \log n + 4)] - 21
\end{equation*}
in the case of $n \geq11$.

Recently the author \cite{Dimitrov} showed
that  for any fixed $1<c<35/34$,  every sufficiently large real number $N$
and a small constant $\varepsilon>0$, the diophantine inequality
\begin{equation*}
|p_1^c+p_2^c+p_3^c-N|<\varepsilon
\end{equation*}
has a solution in primes $p_1,\,p_2,\,p_3$ near to squares.

Motivated by these results and using the method in  \cite{Dimitrov}
in this paper we shall prove the following theorem.
\begin{theorem} Let $c$ be fixed with $1<c<37/36$
and $\delta>0$ be a fixed sufficiently small number.
Then for every sufficiently large positive integer $N$ the diophantine equation
\begin{equation*}
[p^c_1]+[p^c_2]+[p^c_3]=N,
\end{equation*}
is solvable in prime numbers $p_1,p_2,p_3$ such that
\begin{equation*}
\|\sqrt{p_1}\|,\; \|\sqrt{p_2}\|,\; \|\sqrt{p_3}\|<N^{-\frac{6}{17c}\big(\frac{37}{36}-c\big)+\delta}
\end{equation*}
(as usual, $\|\alpha\|$ denotes the distance from $\alpha$ to the nearest integer).
\end{theorem}

\section{Notations}
\indent

Let $N$ be a sufficiently large positive integer and $X=N^{1/c}$.
By $\varepsilon$ we denote an arbitrary small positive number, not the same in all appearances.
The letter $p$  with or without subscript will always denote prime number.
By $\delta$ we denote an fixed sufficiently small positive number.
We denote by $\Lambda(n)$ von Mangoldt's function.
Moreover $e(y)=e^{2\pi \imath y}$.
As usual $[t]$ and $\{t\}$ denote the integer part, respectively, the
fractional part of $t$.
We recall that $t=[t]+\{t\}$ and $\|t\|=\min(\{t\}_,1-\{t\})$.
Let $c$ be fixed with $1<c<37/36$.

Denote
\begin{align}
\label{r}
&r=[\log X]\,;\\
\label{Y}
&Y=X^{-\frac{6}{17}\big(\frac{37}{36}-c\big)+\delta}\,;\\
\label{Delta}
&\Delta=Y/5\,;\\
\label{M}
&M=\Delta^{-1}r.
\end{align}

\section{Preliminary lemmas}
\indent

\begin{lemma}\label{Periodicfunction} Let $r\in \mathbb{N}$.
There exists a function $\chi(t)$ which is $r$-times continuously differentiable and
1-periodic with a Fourier series of the form
\begin{equation}\label{Fourierseries}
\chi(t)=\frac{9}{5}Y+\sum\limits_{m=-\infty\atop{m\neq0}}^\infty g(m) e(mt),
\end{equation}
where
\begin{equation}\label{gmest}
|g(m)|\leq\min\bigg(\frac{1}{\pi|m|},\frac{1}{\pi |m|}
\bigg(\frac{r}{\pi |m|\Delta}\bigg)^r\bigg)
\end{equation}
and
\begin{equation}\label{chit}
\chi(t) =
  \begin{cases}
    1  \quad   \text{ if }  &\|t\|\leq Y-\Delta,\\
    0 \quad  \text{ if }  &\|t\|\geq Y,\\
     \text{between} &0 \, \text{ and }\, 1 \text{ for the other t }.
   \end{cases}
\end{equation}
\end{lemma}
\begin{proof}
See (\cite{Karatsuba}, p. 14).
\end{proof}

\begin{lemma}\label{Buriev} Let $x,y\in\mathbb{R}$ and $H\geq3$.
Then the formula
\begin{equation*}
e(-x\{y\})=\sum\limits_{|h|\leq H}c_h(x)e(hy)+\mathcal{O}\left(\min\left(1, \frac{1}{H\|y\|}\right)\right)
\end{equation*}
holds. Here
\begin{equation*}
c_h(x)=\frac{1-e(-x)}{2\pi i(h+x)}\,.
\end{equation*}
\end{lemma}
\begin{proof}
See (\cite{Buriev}, Lemma 12).
\end{proof}

\section{Outline of the proof}
\indent

Consider the sum
\begin{equation}\label{Gamma}
\Gamma(X)= \sum\limits_{[p^c_1]+[p^c_2]+[p^c_3]=X^c
\atop{\|\sqrt{p_i}\|<Y,\,i=1,2,3}}\log p_1\log p_2\log p_3\,.
\end{equation}
The theorem will be proved if we show that $\Gamma(X)\rightarrow\infty$ as $X\rightarrow\infty$.

From \eqref{chit} and \eqref{Gamma}  we obtain
\begin{equation}\label{Gammaest1}
\Gamma(X)\geq \sum\limits_{[p^c_1]+[p^c_2]+[p^c_3]=X^c}  \prod_{k=1}^{3} \chi(\sqrt{p_k})\log p_k=\int\limits_0^1H^3(\alpha) e(-X^c\alpha)\, d\alpha,
\end{equation}
where
\begin{equation}\label{Halpha}
H(\alpha)=\sum\limits_{p\leq X} \chi(\sqrt{p})e(\alpha [p^c])\log p.
\end{equation}
By \eqref{Fourierseries} and  \eqref{Halpha} we get
\begin{equation}\label{Halphadecomp}
H(\alpha)=\frac{9}{5}YS(\alpha)+V(\alpha),
\end{equation}
where
\begin{align}
\label{Salpha}
&S(\alpha)=\sum\limits_{p\leq X}e(\alpha [p^c])\log p,\\
\label{Valpha}
&V(\alpha)=\sum\limits_{m=-\infty\atop{m\neq0}}^\infty g(m)
\sum\limits_{p\leq X}e(\alpha [p^c]+m\sqrt{p})\log p.
\end{align}
Bearing in mind \eqref{Gammaest1} and \eqref{Halphadecomp} we find
\begin{align}\label{Gamma1est2}
\Gamma(X)&\geq \int\limits_0^1\bigg(\frac{9}{5}YS(\alpha)+V(\alpha) \bigg)^3
e(-X^c\alpha)\,d\alpha\nonumber\\
&=  \bigg(\frac{9}{5}Y\bigg)^3I
+\mathcal{O}\Bigg(Y^2\int\limits_0^1 |S^2(\alpha) V(\alpha)| \,d\alpha  \Bigg)\nonumber\\
&+\mathcal{O}\Bigg(Y\int\limits_0^1 |S(\alpha) V^2(\alpha)| \,d\alpha  \Bigg)
+\mathcal{O}\Bigg(\int\limits_0^1 |V^3(\alpha)| \,d\alpha  \Bigg),
\end{align}
where
\begin{equation*}
I=\int\limits_0^1S^3(\alpha) e(-X^c\alpha)\, d\alpha.
\end{equation*}
According to the asymptotic formula \eqref{RNasymptoticformula}
\begin{equation}\label{Intlowerbound}
I\gg X^{3-c}.
\end{equation}
We have
\begin{equation}\label{Firstint}
\int\limits_0^1 |S^2(\alpha) V(\alpha)| \,d\alpha
\ll\max\limits_{0\leq\alpha\leq 1} |V(\alpha)|\int\limits_0^1|S(\alpha)|^2\,d\alpha.
\end{equation}
Arguing as in (\cite{Tolev}, Lemma 7) for the sum $S(\alpha)$ denoted by \eqref{Salpha}
we obtain
\begin{equation}\label{IntSalphaest}
\int\limits_0^1 |S(\alpha)|^2\,d\alpha\ll X^{1+\varepsilon}.
\end{equation}
Using Cauchy's inequality we get
\begin{equation}\label{Secondint}
\int\limits_0^1 |S(\alpha) V^2(\alpha)| \,d\alpha
\ll\max\limits_{0\leq\alpha\leq 1} |V(\alpha)|
\left(\int\limits_0^1 |S(\alpha)|^2\,d\alpha\right)^{1/2}
\left(\int\limits_0^1 |V(\alpha)|^2\,d\alpha\right)^{1/2}.
\end{equation}
Proceeding as in (\cite{Dimitrov}, Lemma 6) for the sum $V(\alpha)$ denoted by \eqref{Valpha}
we conclude
\begin{equation}\label{IntValphaest}
\int\limits_0^1 |V(\alpha)|^2\,d\alpha\ll X^{1+\varepsilon}.
\end{equation}
Finally
\begin{equation}\label{Thirdint}
\int\limits_0^1 |V(\alpha)|^3\,d\alpha
\ll\max\limits_{0\leq\alpha\leq 1} |V(\alpha)|\int\limits_0^1 |V(\alpha)|^2\,d\alpha.
\end{equation}
In order to complete the proof of the theorem it remains to
find the upper bound in the interval $[0, 1]$ for the sum $V(\alpha)$ denoted by \eqref{Valpha}.

\section{Upper bound of $V(\alpha)$}
\indent

\begin{lemma}\label{ValphaUpperbound}
For the sum $V(\alpha)$ denoted by \eqref{Valpha}  the upper bound
\begin{align}\label{Valphaest}
\max\limits_{0\leq\alpha\leq 1} |V(\alpha)|\ll
&\Big( M^{1/2}X^{7/12} + M^{1/6} X^{3/4}+X^{11/12}+X^{\frac{2c+31}{34}}\nonumber\\
&+M^{1/4}X^{\frac{69-12c}{68}}+M^{1/12}X^{\frac{131-8c}{136}}
+X^{\frac{32c+3}{68}}\Big)X^\varepsilon
\end{align}
holds.
\end{lemma}
\begin{proof}
Let $0\leq\alpha\leq 1$.
Denote
\begin{equation}\label{Ualpha}
U(\alpha, m)=\sum\limits_{p\leq X} e(\alpha [p^c]+m\sqrt{p})\log p.
\end{equation}
From  \eqref{r}, \eqref{Delta}, \eqref{M}, \eqref{gmest},
 \eqref{Valpha} and \eqref{Ualpha} it follows
\begin{align}\label{Valphaest1}
|V(\alpha)|&\ll\sum_{0<|m|\leq M}\frac{1}{|m|}|U(\alpha, m)|
+X\sum_{|m|>M}|g(m)|\nonumber\\
&\ll\sum_{0<|m|\leq M}\frac{1}{|m|}|U(\alpha, m)| +\bigg(\frac{r}{\pi M\Delta}\bigg)^rX\nonumber\\
&\ll \sum_{0<|m|\leq M}\frac{1}{|m|}|U(\alpha, m)|+1.
\end{align}
By \eqref{Ualpha} and Lemma \ref{Buriev} with $x=\alpha$ and $y=n^c$ we obtain
\begin{align}\label{Ualphaest1}
U(\alpha, m)&=\sum\limits_{n\leq X}\Lambda(n)e(\alpha n^c+m\sqrt{n})e(-\alpha\{n^c\})+\mathcal{O}( X^{1/2})\nonumber\\
&=\sum\limits_{|h|\leq H}c_h(\alpha)\sum\limits_{n\leq X}\Lambda(n)e((h+\alpha)n^c+m\sqrt{n})\nonumber\\
&+\mathcal{O}\left((\log X)\sum\limits_{n\leq X}\min\left(1, \frac{1}{H\|n^c\|}\right)\right).
\end{align}
Now \eqref{Valphaest1} and \eqref{Ualphaest1} imply
\begin{equation}\label{Valphaest2}
|V(\alpha)|\ll V_1(\alpha)+X^\varepsilon V_2,
\end{equation}
where
\begin{align*}
&V_1(\alpha)= \sum_{0<|m|\leq M}\frac{1}{|m|}
\bigg|\sum\limits_{|h|\leq H}c_h(\alpha)
\sum\limits_{n\leq X}\Lambda(n)e((h+\alpha)n^c+m\sqrt{n})\bigg|,\\
&V_2=\sum\limits_{n\leq X}\min\left(1, \frac{1}{H\|n^c\|}\right).
\end{align*}
Working similar to \cite{Dimitrov} we get
\begin{align}\label{V1alphaest}
\max\limits_{0\leq\alpha\leq 1} |V_1(\alpha)|\ll
&\Big( M^{1/2}X^{7/12} + M^{1/6} X^{3/4}+X^{11/12}
+H^{1/16}X^{\frac{2c+29}{32}}\nonumber\\
&+H^{-3/16}M^{1/4}X^{\frac{33-6c}{32}}
+H^{-1/16}M^{1/12}X^{\frac{31-2c}{32}}\Big)X^\varepsilon.
\end{align}
Arguing as in Cai \cite{Cai} we find
\begin{equation}\label{V2est}
 |V_1|\ll (H^{-1}X+H^{1/2}X^{c/2})X^\varepsilon.
\end{equation}
Summarizing \eqref{Valphaest2} --  \eqref{V2est}
and choosing
\begin{equation*}
H=X^{\frac{3-2c}{34}}
\end{equation*}
we obtain the estimation \eqref{Valphaest}.
\end{proof}

\section{Proof of the Theorem}
\indent

Using \eqref{Gamma1est2},  \eqref{Firstint}, \eqref{IntSalphaest},  \eqref{Secondint},
\eqref{IntValphaest},  \eqref{Thirdint}  and  Lemma \ref{ValphaUpperbound}
we obtain
\begin{align}\label{Gammaest3}
\Gamma(X)\geq\bigg(\frac{9}{5}Y\bigg)^3I  &+\mathcal{O} \Big( \big(M^{1/2}X^{19/12}
+ M^{1/6} X^{7/4}+X^{23/12}+X^{\frac{2c+65}{34}}\nonumber\\
&+M^{1/4}X^{\frac{137-12c}{68}}+M^{1/12}X^{\frac{267-8c}{136}}
+X^{\frac{32c+71}{68}}\big)X^\varepsilon\Big).
\end{align}
From  \eqref{Y}, \eqref{Delta}, \eqref{M}, \eqref{Intlowerbound}, \eqref{Gammaest3}
and  choosing   $\varepsilon<\delta$  we find
\begin{equation}\label{Gammaest4}
\Gamma(X)\gg Y^3X^{3-c}.
\end{equation}
Bearing in mind \eqref{Y} and \eqref{Gammaest4}
we establish that $\Gamma(X)\rightarrow\infty$ as $X\rightarrow\infty$.

The proof of the Theorem is complete.

\vskip18pt
\footnotesize
\begin{flushleft}
S. I. Dimitrov\\
Faculty of Applied Mathematics and Informatics\\
Technical University of Sofia \\
8, St.Kliment Ohridski Blvd. \\
1756 Sofia, BULGARIA\\
e-mail: sdimitrov@tu-sofia.bg\\
\end{flushleft}

\end{document}